\documentclass[11pt,oneside]{amsart}
\usepackage[foot]{amsaddr}
\usepackage{bm, calc,geometry, verbatim, graphicx, amssymb, color, mathabx, mathtools, enumitem, bm, mathrsfs, amsmath, amsthm, ulem, xspace,tabularx,leftindex}
\usepackage[usenames,dvipsnames,svgnames,table]{xcolor}
\usepackage[pdftex,bookmarks,colorlinks,breaklinks]{hyperref}  
\usepackage{rotating}
\usepackage{adjustbox}
\usepackage{blindtext}
\usepackage{relsize}

\hypersetup{linkcolor=blue,citecolor=red,filecolor=dullmagenta,urlcolor=darkblue} 
\newtheorem{theorem}{Theorem}[section]
\newtheorem{corollary}[theorem]{Corollary}
\newtheorem{lemma}[theorem]{Lemma}

\newtheorem{question}[theorem]{Problem}

\newcommand\mult{\operatorname{\textup{{\fontfamily{ptm}\selectfont mult}}}}
\newcommand\dg{\operatorname{\textup{{\fontfamily{ptm}\selectfont deg}}}}

\newcommand\brho{\operatorname{\boldsymbol{\rho}}}

\newcommand\roundup[1]{\left\lceil#1\right\rceil}
\newcommand\rounddown[1]{\left\lfloor#1\right\rfloor}

      \makeatletter
      \def\@setcopyright{}
      \def\serieslogo@{}
      \makeatother      

\begin{document}
\author{Amin  Bahmanian}
   \address{Department of Mathematics,
  Illinois State University, Normal, IL USA 61790-4520}

   \author{Anna Johnsen-Yu}
  \address{Department of Mathematics and Statistics,  Georgia State University, Atlanta, GA USA 30303-2918}
  \address{Department of Mathematics, Vanderbilt University, Nashville, TN USA 37240-0001}

\title[Embedding Equitable Rectangles]{Embedding  Equitable Rectangles}
\dedicatory{In memory of Charles Curtis Lindner (1938 --  2023)}

\begin{abstract}
Completing partial Latin squares is NP-complete. Ryser characterized precisely when an \(r\times s\) Latin rectangle can be completed to an \(n\times n\) Latin square. We extend this theorem to a broader setting. Let \(M\) be an \(n\times n\) array whose top-left \(r\times s\) subarray is filled with symbols from \(\{1,2,\dots,k\}\), where \(1\le k\le n^2\). We determine necessary and sufficient conditions under which the remaining cells of \(M\) can be filled so that each symbol \(\ell\in\{1,2,\dots,k\}\) appears exactly \(\rho_\ell\) times in total, while the numbers of occurrences of each symbol in any two rows and in any two columns differ by at most one.

Equivalently, our result characterizes when a partial edge-coloring of \(K_{r,s}\) can be extended to an edge-coloring of \(K_{n,n}\) in which each color class is spanning and almost regular, with prescribed color-class sizes. In this sense, our theorem generalizes Baranyai's construction of almost-regular colorings of complete uniform multipartite hypergraphs, restricted to the bipartite case.

Our theorem unifies and generalizes several classical results. When \(s=k=\rho_1=\cdots=\rho_k=n\), it reduces to Hall's theorem, and when \(k=\rho_1=\cdots=\rho_k=n\), it yields Ryser's completion theorem for Latin rectangles. Additional special cases recover results of Goldwasser, Hilton, Hoffman, and Özkan, as well as a theorem of Bahmanian. Thus, our work may be viewed as a common generalization of these completion theorems and Baranyai's theorem.
\end{abstract}

   \subjclass[2010]{05B15, 05C70, 05C65, 05C15}
   \keywords{Equitable rectangles, Latin Squares, Embedding, Ore's Theorem, Ryser's Theorem, Edge-coloring}

   \maketitle   
\section{\normalfont{Introduction}}

Throughout this paper, for $x,y\in\mathbb{R}$, $x\approx y$ means $\rounddown{y}\leq x \leq \roundup{y}$, 
 \begin{align*}
    &&
    r,s, n,k\in \mathbb N,
    &&
    \min\{r,s\}< n,
    &&
    \max\{r,s\}\leq n,
    &&
    1\leq k\leq n^2,
    &&
    [k]:=\{1,2,\dots, k\},
    &&
\end{align*}
and  $\brho$ is a $k$-tuple of integers satisfying the following conditions. 
 \begin{align*}
    &&
    \brho=(\rho_1,\dots,\rho_k),
    &&
    1\leq \rho_1,\dots,\rho_k\leq n^2,
    &&
    \displaystyle\sum\nolimits_{\ell\in [k]} \rho_\ell= n^2.
    &&
  \end{align*}
In an $r\times s$ (multi-)array $M$ on a set $[k]$ of {\it symbols}, each cell  contains an element (multisubset, respectively) of $[k]$ for $i\in[r], j\in [s]$, where $[r], [s]$ are the set of rows and columns, respectively. 
For $\ell\in [k]$, $i\in [r]$, and $j\in [s]$, $|M_\ell|$, $|M_\ell^i|$, and $|\leftindex^j M_\ell|$ are the number of occurrences of $\ell$ in $M$, in row $i$ of $M$, and in column $j$ of $M$, respectively. In a multi-array $M$, $\,^{j}M^{i}$ denotes the multiset in cell $(i,j)$ and $\lvert\,^{j}M^{i}\rvert$ its cardinality; if $i$ or $j$ exceeds the dimensions of $M$, we set $\lvert\,^{j}M^{i}\rvert = 0$.

An $r \times s$ {\it equitable $\brho$-Latin rectangle} $M$ is  an $r\times s$  array on $[k]$ such that for $\ell\in [k], i\in [r], j\in [s]$, 
\[
\begin{aligned}
    |M_\ell^i| &\le \roundup{\dfrac{\rho_\ell}{n}}, &
    |\leftindex^j M_\ell| &\le \roundup{\dfrac{\rho_\ell}{n}}, &
    |M_\ell| &\le \rho_\ell, \\
    |M_\ell^i| &\ge \rounddown{\dfrac{\rho_\ell}{n}} \text{ if } r=n, &
    |\leftindex^j M_\ell| &\ge \rounddown{\dfrac{\rho_\ell}{n}}  \text{ if } s=n, &
    |M_\ell| &= \rho_\ell \text{ if } r=s=n.
\end{aligned}
\]
If $r=s=n$, then  each symbol \(\ell\) appears approximately equally often in every row and column of $M$, with total occurrence \(\rho_\ell\), and $M$ is an {\it equitable $\brho$-Latin square}.  
If \(\rho_\ell\le n\) for  \(\ell\in[k]\), then an equitable \(\brho\)-Latin square (respectively, rectangle) is simply called a {\it \(\brho\)-Latin square} (respectively, {\it \(\brho\)-Latin rectangle}). In the special case \(\rho_\ell=n\) for  \(\ell\in[k]\), we recover the usual notions of a {\it Latin square} and {\it Latin rectangle}.

 A Latin square is equivalent to a quasigroup. Indeed, if \(L\) is an
\(n\times n\) Latin square on a set \(Q\), we define a binary operation
$\circ:Q\times Q\to Q$
by
$x\circ y=L(x,y)$.
The Latin property implies that for  \(a,b\in Q\), the equations
\[
a \circ x=b
\qquad\text{and}\qquad
y \circ a=b
\]
have unique solutions \(x,y\in Q\). More generally, an equitable $\brho$-Latin square corresponds to an {\it equitable $\brho$-quasigroup},  a triple $(Q,S,\circ)$,
where $|Q|=n$, $S=[k]$, and $\circ:Q\times Q\to S$ such that for  $a\in Q$ and $b\in S$, the equation
\[
x \circ y=b
\]
has exactly $\rho_b$ ordered solutions, and each of the equations
\[
x \circ a=b
\qquad\text{and}\qquad
a \circ y=b
\]
has either $\lfloor \rho_b/n \rfloor$ or $\lceil \rho_b/n \rceil$ solutions.

We characterize precisely when an $r\times s$ equitable $\brho$-Latin rectangle can be completed to an $n\times n$ equitable $\brho$-Latin square. This describes when a partial equitable $\brho$-quasigroup can be embedded into a full one. Before stating our main result, Theorem~\ref{main2}, we introduce the necessary background and notation. The surveys of Lindner \cite{MR379211} and Rodger \cite{MR1275861} contain broad accounts of the theory of partial Latin squares, including many open problems that continue to attract attention. Here, we restrict ourselves to recalling several classical results relevant to our work.

The problem of completing partial Latin squares is NP-complete \cite{MR739595}. In 1945, Hall established the following result.
\begin{theorem} \cite{MR13111} \label{thm:hall}
    Any $r\times n$ Latin rectangle can be extended to an $n\times n$ Latin square.
\end{theorem}
Ryser later generalized Hall's theorem as follows.
\begin{theorem} \cite{MR42361} \label{thm:ryser}
    Any $r\times s$ Latin rectangle $M$ can be extended to an $n\times n$ Latin square if and only if $|M_\ell|\geq r+s-n$ for $\ell\in[n]$.
\end{theorem}

Building on previous generalizations of Hall's and Ryser's theorems for $\brho$-Latin rectangles \cite{MR3280683,BAHMANIAN2022105632}, 
our main result extends these classical combinatorial theorems to the more general setting of equitable $\brho$-Latin rectangles, 
capturing both the global symbol counts and their approximately uniform distribution across rows and columns. 
Corollary~\ref{main1intro} illustrates the resulting generalization of Hall's theorem, 
while the comment before Corollary ~\ref{cor:ndivrhoiryser} explains how Theorem~\ref{main2} recovers Ryser's theorem in this context.

In order to formulate our main theorem, we first introduce the concepts of a saturated symbol and a feasible sequence. 
A symbol $\ell \in [k]$ is called \textit{free} if $n \nmid \rho_\ell$, and \textit{forced} otherwise.  
We say that $\ell$ is \textit{saturated} in a row $i \in [r]$ or column $j \in [s]$ of $M$ if it appears  $\lceil \rho_\ell / n \rceil$ times in that row or column; equivalently, we also say that the row $i$ or column $j$ of $M$ is \textit{$\ell$-saturated}. For a free symbol $\ell$, let $\eta_I(\ell)$ and $\eta_J(\ell)$ be the number of $\ell$-unsaturated rows in $I \subseteq [r]$ and $\ell$-unsaturated columns in $J \subseteq [s]$, respectively; for forced symbols, we set $\eta_I(\ell) = \eta_J(\ell) = 0$. 
For $i\in[r], j\in[s]$, we abbreviate $\eta_{\{i\}}(\ell)$ and $\eta_{\{j\}}(\ell)$ to $\eta_i(\ell)$ and $\eta_j(\ell)$, respectively, and for $K\subseteq[k]$, we define $\eta_K(i)$ and $\eta_K(j)$ to be the number of free symbols $\ell\in K$ that are unsaturated in row $i$ and column $j$, respectively.
For notational convenience, we typically use \(r\) and \(s\); when it is helpful to combine two equations, we write \(r_1 = r\) and \(r_2 = s\). A sequence $\{(a_{\ell1},a_{\ell2})\}_{\ell\in[k]}$ of non-negative integer pairs is
\textit{feasible} (with respect to an $r_1 \times r_2$ equitable $\boldsymbol{\rho}$-Latin rectangle $M$) if for $t=1,2$ we have
\[
\left\{
\begin{aligned}
& a_{\ell t} \;\ge\; \rho_\ell - r_t \left\lceil \frac{\rho_\ell}{n} \right\rceil
  - (n-r_t)\left\lfloor \frac{\rho_\ell}{n} \right\rfloor,\\
& a_{\ell t} 
  \;\le\; \min\Big\{n\roundup{\dfrac{\rho_\ell}{n}}, \rho_\ell + \eta_{[r_t]}(\ell)\Big\} - r_t\roundup{\dfrac{\rho_\ell}{n}} - (n-r_t) \rounddown{\dfrac{\rho_\ell}{n}},
  \\[1.5mm]
& a_{\ell 1} + a_{\ell 2} 
  \;\ge\; \rho_\ell - |M_\ell| 
  - (2n - r_1 - r_2) \left\lfloor \frac{\rho_\ell}{n} \right\rfloor, \\[1.5mm]
& \sum_{\ell \in [k]} a_{\ell t} 
  \;=\; (n-r_t) \Bigl( n - \sum_{\ell \in [k]} \left\lfloor \frac{\rho_\ell}{n} \right\rfloor \Bigr).
\end{aligned}
\right.
\]
Let
\[
f_t(\ell) =
r_t \left\lceil \dfrac{\rho_\ell}{n} \right\rceil
+ (n-r_t) \left\lfloor \dfrac{\rho_\ell}{n} \right\rfloor
- \rho_\ell + a_{\ell t}
\qquad t=1,2.
\]

\begin{theorem}\label{main2}
An $r_1\times r_2$ equitable $\brho$-Latin rectangle $M$ extends to an $n\times n$ equitable $\brho$-Latin square if and only if there exists a feasible sequence $\{(a_{\ell1}, a_{\ell2})\}_{\ell\in[k]}$  such that for $t=1,2$ we have
\begin{align}
\rho_\ell - |M_\ell|
&\;\ge\; (n-r_t) \rounddown{\dfrac{\rho_\ell}{n}}
&\quad &\text{for } \ell \in [k], 
\label{MainNcond} \\[1em]
n|Z| 
&\;\ge\; \sum_{\ell \in [k]} \Big(|Z|\Bigl\lceil \frac{\rho_\ell}{n} \Bigr\rceil -\min \bigl\{ f_t(\ell),\, \eta_Z(\ell) \bigr\}\Big)
&\quad &\text{for }  Z \subseteq [r_t].
\label{ncond2R}
\end{align}
\end{theorem}

Readers may find the following special case of Theorem \ref{main2} appealing due to its simplicity and independence from feasible sequences. 
\begin{corollary}\label{main1intro}
    An $r \times n$ equitable $\boldsymbol{\rho}$-Latin rectangle $M$ extends to an $n \times n$ equitable $\boldsymbol{\rho}$-Latin square  if and only if the following conditions are met.
\begin{align*}
     &\dfrac{\rho_\ell - |M_\ell|}{n-r} \approx \dfrac{\rho_\ell}{n} &\text{for } \ell\in [k],\\
     &n|J| \geq \sum_{\ell\in[k]}\Big(|J|\roundup{\dfrac{\rho_\ell}{n}} - \min\left\{n\roundup{\dfrac{\rho_\ell}{n}} - \rho_\ell, \eta_J(\ell)\right\} \Big) &\text{for } J\subseteq[s].
\end{align*}
\end{corollary} 

The proof of Theorem~\ref{main2} uses Ore’s theorem \cite{MR83725} together with an integer-making lemma due to Nash-Williams \cite{MR916377}. Baranyai \cite{MR535941} studied almost-regular colorings of multipartite hypergraphs, 
and Theorem~\ref{main2} generalizes this result for bipartite graphs by providing conditions under which an edge coloring of the complete bipartite graph $K_{r,s}$ can be extended to an edge coloring of $K_{n,n}$ such that each color class is spanning, almost regular, and the number of edges of each color is prescribed. 
Unlike previous embedding results that restrict the maximum number of occurrences of each symbol to at most $n$, Theorem~\ref{main2} generalizes corresponding results for $\brho$-Latin squares \cite{BAHMANIAN2022105632} 
and $r\times n$ $\brho$-Latin rectangles \cite{MR3280683} to equitable $\brho$-Latin squares and rectangles, as illustrated in Corollary~\ref{main1}.  We also note that equitable $\brho$-Latin squares can be viewed as a natural generalization of exact $(p,p,1)$-Latin squares \cite{ANDERSEN1980125,ANDERSEN1980235}, 
where each symbol occurs exactly $p$ times in each row and  column. 
In the case of equitable $\brho$-Latin squares, the total occurrences of each symbol may vary or may not be divisible by $n$, 
and the occurrences are distributed as evenly as possible across rows and columns. 
Orthogonality of special equitable rectangles was studied by Asplund and Keranen \cite{MR2787313}.

The paper is organized as follows. Section~\ref{s:ryser} contains a proof of Theorem~\ref{main2}. A sketch of the proof of the sufficiency part of Theorem~\ref{main2} is given at the beginning of Subsection~\ref{suffproofsubsec}. Section~\ref{s:corollaries} presents consequences of our main results along with several open problems.

\section{Proof of Theorem \ref{main2}}  \label{s:ryser}
\subsection{Tools}
All graphs will be loopless, but may contain multiple edges. For a graph $G$, vertices $u,v\in V(G)$, and subsets $S,T\subseteq V(G)$, we denote by $\dg_G(u)$, $\mult_G(uv)$, and  $\mult_G(uS)$ the number of edges incident with $u$,  the number of edges whose endpoints are $u$ and $v$, and the number of edges incident to $u$ with an endpoint in $S$, respectively. For a real-valued function $f$ on a domain $D$ and  $S\subseteq D$, $f(S) :=\sum_{x\in S}f(x)$. For  a non-negative integer function $f$  on the vertex set of a graph $G$, an {\it $f$-factor} is a spanning subgraph $F$ of $G$ with the property that $\dg_F(x)=f(x)$ for each $x$. By Ore's Theorem \cite{MR83725}, 
the bipartite graph $G[X,Y]$ has an $f$-factor if and only if $f(X)=f(Y)$ and
\begin{align*} 
    f(A)\leq \sum_{u\in Y} \min\Big\{f(u), \mult_G(uA)\Big\} \quad \text {for } A\subseteq X.
\end{align*}

Recall that  $x\approx y$ means $\rounddown{y}\leq x \leq \roundup{y}$. Note that this is a transitive relation and that if $x\approx y$ and $n\in\mathbb{N}$, then we have $x/n \approx y/n$. 
A family of sets is  {\it laminar} if any two sets in the family are either disjoint or one is contained in the other. Nash-Williams' lemma \cite{MR916377} states that for $n\in \mathbb{N}$ and two laminar families $\mathscr{A}$ and $\mathscr{B}$   of subsets of a finite set $S$, there exists a subset $U \subseteq S$ such that
\[
|U \cap W| \approx \frac{|W|}{n} \quad \text {for } W \in \mathscr{A} \cup \mathscr{B}.
\]

\subsection{Proof of Necessity}
Let \(N\) be an \(n\times n\) equitable \(\brho\)-Latin square containing an equitable \(\brho\)-Latin rectangle \(M\) in its top-left \(r\times s\) subrectangle, and let \(\overline{M} := N \setminus M\). We refer to the last \(n-r\) rows and the last \(n-s\) columns of \(N\) as the {\it rows} and {\it columns} of \(\overline{M}\), respectively.

Let  $\ell\in[k]$. Since \(N\) is an equitable \(\brho\)-Latin square, every row and column of \(N\) contains at least \(\lfloor \rho_\ell/n \rfloor\) copies of \(\ell\). In particular,
\begin{equation}\label{lastrowcol}
|N_\ell^i| \ge \left\lfloor \frac{\rho_\ell}{n}\right\rfloor \text{ for } i=r+1,\dots,n,
\qquad
|\leftindex^j N_\ell| \ge \left\lfloor \frac{\rho_\ell}{n}\right\rfloor \text{ for } j=s+1,\dots,n.
\end{equation}
Exactly $\rho_\ell-|M_\ell|$ occurrences of $\ell$ lie in $\overline{M}$, so by \eqref{lastrowcol},
\begin{align*}
    \rho_\ell - |M_\ell|
    \ge
    \max\{n-r, n-s\}
    \left\lfloor\dfrac{\rho_\ell}{n}\right\rfloor.
\end{align*}
If $\ell$ is free, let $a_\ell$ and $b_\ell$ denote the number of $\ell$-saturated rows  and  columns of $\overline{M}$, respectively; if $\ell$ is forced, we set $a_\ell=b_\ell=0$. Then
\begin{align*}
a_\ell \le (n-r) \Bigl( \roundup{\dfrac{\rho_\ell}{n}} - \rounddown{\dfrac{\rho_\ell}{n}} \Bigr), \;\;
b_\ell \le (n-s) \Bigl( \roundup{\dfrac{\rho_\ell}{n}} - \rounddown{\dfrac{\rho_\ell}{n}} \Bigr).
\end{align*}
Recall that $\overline{M}$ contains $\rho_\ell-|M_\ell|$ occurrences of $\ell$. 
By \eqref{lastrowcol}, at most
$(2n-r-s)\rounddown{\rho_\ell/n}$
occurrences of $\ell$ can lie in $\ell$-unsaturated  rows and columns of $\overline{M}$ . 
All remaining occurrences must therefore lie in $\ell$-saturated rows or columns of $\overline{M}$, implying
\begin{align*}
    \rho_\ell - |M_\ell|\leq a_\ell + b_\ell   + (2n-r-s)\left\lfloor\dfrac{\rho_\ell}{n}\right\rfloor.
\end{align*}
The rows of $\overline{M}$ contain $n(n-r)$ cells in total, and the columns of $\overline{M}$ contain $n(n-s)$ cells in total. Again, using \eqref{lastrowcol}, we account for
$(n-r)\left\lfloor \rho_\ell / n \right\rfloor$ and $(n-s)\left\lfloor \rho_\ell / n \right\rfloor$
copies of $\ell$ in $\overline{M}$ coming from row and column constraints, respectively.
The remaining copies of $\ell$ in $\overline{M}$ are precisely those counted by $a_\ell$ and $b_\ell$, and hence
\begin{align*}
\sum_{\ell\in[k]} \left(a_\ell + (n-r)\rounddown{\dfrac{\rho_\ell}{n}}\right) = n(n-r), \quad
\sum_{\ell\in[k]} \left(b_\ell + (n-s)\rounddown{\dfrac{\rho_\ell}{n}}\right) = n(n-s).
\end{align*}

Let \(I\subseteq [r]\), viewed as rows of both \(M\) and \(N\). We count in two ways $\mathfrak{u}$ where
\[
\mathfrak{u}:=\sum_{\ell\in [k]} |\mathscr U_I(\ell)|,
\qquad
\mathscr U_I(\ell):=\bigl\{ (i,\ell)\mid i\in I,\  |N_\ell^i| < \left\lceil \frac{\rho_\ell}{n}\right\rceil \bigr\}
\quad \text{for } \ell\in[k].
\]
On the one hand, each row of \(I\) contains \(n\) cells, while each symbol \(\ell\in[k]\) may occur in a row at most
\(\roundup{\rho_\ell/n}\) times. Hence
\(\sum_{\ell\in[k]}\roundup{\rho_\ell/n} - n\)
counts the number of symbols in \([k]\) that are unsaturated in a given row of \(N\). So
\[
\mathfrak{u}=|I|\Bigl(\sum_{\ell\in[k]}\left\lceil \frac{\rho_\ell}{n}\right\rceil-n\Bigr).
\]
On the other hand, by the definition of \(\eta_I(\ell)\), we have $|\mathscr U_I(\ell)|\le \eta_I(\ell)$ for  \(\ell\in[k]\). The number of occurrences of \(\ell\) in the last \(n-r\) rows of \(N\) is
\((n-r)\rounddown{\rho_\ell/n}+a_\ell\), and consequently the number of occurrences of \(\ell\) in the first \(r\) rows is
\(\rho_\ell - (n-r)\rounddown{\rho_\ell/n}- a_\ell\). Since each row in \([r]\) can contain at most
\(\roundup{\rho_\ell/n}\) copies of \(\ell\), the total number of available positions for \(\ell\) in rows of \([r]\) is at most
\(r\roundup{\rho_\ell/n}\). As the total number of available positions for $\ell$ in rows of $[r]$ cannot be less than the actual number of occurrences of $\ell$ in rows of $[r]$, it follows immediately that
\begin{align*}
    a_\ell \geq \rho_\ell - (n-r)\rounddown{\dfrac{\rho_\ell}{n}} - r\roundup{\dfrac{\rho_\ell}{n}}.
\end{align*}
Comparing the total number of available positions for \(\ell\) with the actual number of occurrences of \(\ell\) in the first \(r\) rows further yields
$|\mathscr U_I(\ell)| \le |\mathscr U_{[r]}(\ell)| = r\left\lceil \rho_\ell/n \right\rceil + (n-r)\left\lfloor \rho_\ell/n \right\rfloor + a_\ell - \rho_\ell$,
and therefore
\[
|\mathscr U_I(\ell)|
\le
\min\left\{\eta_I(\ell),\;
r\roundup{\frac{\rho_\ell}{n}} + (n-r)\rounddown{\frac{\rho_\ell}{n}} + a_\ell - \rho_\ell
\right\}.
\]
Consequently, summing over \(\ell\in[k]\), we obtain
\begin{align*}
|I|\Big(\sum_{\ell\in[k]}\roundup{\dfrac{\rho_\ell}{n}} - n\Big)
\le
\sum_{\ell\in[k]}
\min\left\{
r\roundup{\frac{\rho_\ell}{n}} + (n-r)\rounddown{\frac{\rho_\ell}{n}} + a_\ell - \rho_\ell,\;
\eta_I(\ell)
\right\}.
\end{align*}
As $|\mathscr U_{[r]}(\ell)|\leq \eta_{[r]}(\ell)$, we also have that
\begin{align*}
    a_\ell \leq \eta_{[r]}(\ell) + \rho_\ell - r\roundup{\dfrac{\rho_\ell}{n}}  - (n-r)\rounddown{\dfrac{\rho_\ell}{n}}.
\end{align*}
By a similar argument, 
\begin{align*}
|J|\Big(\sum_{\ell\in[k]}\roundup{\dfrac{\rho_\ell}{n}} - n\Big)
&\le
\sum_{\ell\in[k]}
\min\left\{
s\roundup{\frac{\rho_\ell}{n}} + (n-s)\rounddown{\frac{\rho_\ell}{n}} + b_\ell - \rho_\ell,\;
\eta_J(\ell)
\right\} &\forall J\subseteq[s],\\
    \rho_\ell - (n-s)\rounddown{\dfrac{\rho_\ell}{n}} - s\roundup{\dfrac{\rho_\ell}{n}} &\leq b_\ell \leq \eta_{[s]}(\ell) + \rho_\ell - s\roundup{\dfrac{\rho_\ell}{n}}  - (n-s)\rounddown{\dfrac{\rho_\ell}{n}} &\forall \ell\in[k].
\end{align*}

\subsection{Proof of Sufficiency} \label{suffproofsubsec}
Assume the necessary conditions hold. Starting from the given equitable \(\brho\)-Latin rectangle   $M$, we first add one column and one row by encoding admissible symbol placements into auxiliary bipartite graphs $\Gamma_1$ and $\Gamma_2$, where edges correspond to occurrences of free symbols that are not yet saturated in the relevant rows and columns. Degree functions $f_1$ and $f_2$ are defined from the remaining symbol requirements, and condition~\eqref{ncond2R} together with feasibility of $(a_\ell,b_\ell)_{\ell\in[k]}$ ensures the hypotheses of Ore’s $f$-factor theorem. Hence $\Gamma_1$ and $\Gamma_2$ admit $f_1$- and $f_2$-factors, which determine a multi-array $T$ satisfying Lemma~\ref{claim1}, including the correct total symbol counts and the prescribed row and column balance conditions.

We then iteratively apply Lemma~\ref{claim2}, where at each step a row or column is split and symbol incidences are represented as a multiset of cell--symbol pairs. The construction is governed by two laminar families arising from symbol classes and position classes, and Nash--Williams’ lemma is applied to select a subset $U$ that simultaneously preserves all required proportional constraints across these families. This ensures that the updated arrays continue to satisfy the same $f$-factor feasibility conditions at each stage, so the induction proceeds.

The process terminates with an $n\times n$ multi-array $P$ satisfying $|P_\ell|=\rho_\ell$ for all $\ell\in[k]$ and exact row and column balance requirements, hence $P$ is an equitable $\rho$-Latin square containing $M$.

\begin{lemma} \label{claim1} 
    There exists a $\min\{n,r+1\}\times \min\{n, s+1\}$      multi-array $T$ on $[k]$ whose top-left $r\times s$ subarray is identical to $M$, and for  $i \in [r],\ j \in [s], \ell \in [k]$, 
\[
\left\{
\begin{aligned}
&|T_\ell| = \rho_\ell,\\[1mm]
&|\,^{s+1} T^i| = n-s, 
  && |\,^j T^{r+1}| = n-r, 
  && |\,^{s+1} T^{r+1}| = (n-r)(n-s),\\[2mm]
&|T^i_\ell| \approx |\,^j T_\ell| \approx \frac{\rho_\ell}{n}, 
  && |T^{r+1}_\ell| \approx (n-r) \frac{\rho_\ell}{n}, 
  && |\,^{s+1} T_\ell| \approx (n-s) \frac{\rho_\ell}{n}.
\end{aligned}
\right.
\]
\end{lemma}
\begin{proof}
If $n>s$, 
then let $\Gamma_1[X,[k]]$  be  a  bipartite graph, where 
$X:=\{x_1,\dots,x_r\}$ and $[k]$ represent the rows and symbols of $M$. For $i\in[r]$ and $\ell\in [k]$, $x_i\ell$ is an edge in $\Gamma_1$ if and only if $\ell$ is a free  symbol that is unsaturated in row $i$ of $M$. Let $f_1$ be a mapping on the vertex set of $\Gamma_1$  such that for $i\in [r],\ell\in [k]$,
\begin{align*}
f_1(x_i) = \sum_{\ell \in [k]} \roundup{\frac{\rho_\ell}{n}} - n, \quad
f_1(\ell) = r \roundup{\frac{\rho_\ell}{n}} +(n- r) \rounddown{\frac{\rho_\ell}{n}}  - \rho_\ell + a_\ell
\end{align*}
Observe that $f_1(x_i)\geq 0$ for $i\in [r]$. The  feasibility of $(a_\ell,b_\ell)_{\ell\in[k]}$ ensures $f_1(\ell)\geq 0$ for $\ell\in [k]$, and that
$\sum_{\ell\in[k]}a_\ell = (n-r)\Big(n-\sum_{\ell\in[k]}\rounddown{\dfrac{\rho_\ell}{n}}\Big)$, and so
\begin{align*}
    \sum_{i\in[r]}\Big(\sum_{\ell\in[k]}\roundup{\dfrac{\rho_\ell}{n}} - n\Big) = \sum_{\ell\in[k]}\left(r\roundup{\dfrac{\rho_\ell}{n}}-r\rounddown{\dfrac{\rho_\ell}{n}} + n\rounddown{\dfrac{\rho_\ell}{n}} - \rho_\ell + a_\ell\right),
\end{align*}
or equivalently, $f_1(X) = f_1([k])$.

We claim that $\dg_{\Gamma_1}(u)\geq f_1(u)$ for $u\in X\cup [k]$. To see this, recall that $\mult_{\Gamma_1}(x_i\ell) = 1$ if $\ell$ is free such that $|M_\ell^i|<\roundup{\rho_\ell/n}$ and is $0$ otherwise. Hence, we first have by \eqref{ncond2R} that for $i\in[r]$,
\begin{align*}
    \dg_{\Gamma_1}(x_i) = \sum_{\ell\in[k]}\eta_i(\ell) &\geq \sum_{\ell\in[k]}\min\{r\roundup{\dfrac{\rho_\ell}{n}} + (n-r)\rounddown{\dfrac{\rho_\ell}{n}} - \rho_\ell + a_\ell,\eta_i(\ell)\}\\
    &\geq |\{x_i\}|\Big(\sum_{\ell\in[k]}\roundup{\dfrac{\rho_\ell}{n}}-n\Big)= \sum_{\ell\in[k]}\roundup{\dfrac{\rho_\ell}{n}}-n = f_1(x_i).
\end{align*}
Second, we have by the feasibility of the sequence $\{(a_\ell, b_\ell)\}_{\ell\in[k]}$ that $a_\ell \leq \rho_\ell + \eta_{[r]}(\ell) - r\roundup{\rho_\ell/n} - (n-r)\rounddown{\rho_\ell/n}$, and so
\begin{align*}
    \dg_{\Gamma_1}(\ell) = \sum_{i\in[r]}\eta_i(\ell) = \eta_{[r]}(\ell) &\geq r\roundup{\dfrac{\rho_\ell}{n}} + (n-r)\rounddown{\dfrac{\rho_\ell}{n}} - \rho_\ell + a_\ell = f_1(\ell).
\end{align*}

Using \eqref{ncond2R}, we  have 
\begin{equation*} 
    f_1(I) \leq \sum_{\ell\in[k]}\min\{f_1(\ell), \mult_{\Gamma_1}(I\ell)\} \quad \text{for } I\subseteq X.
\end{equation*}
Thus, by Ore's theorem, $\Gamma_1$ has  an $f_1$-factor $\Theta_1$.

If $n>r$, we define $\Gamma_2[Y,[k]]$ where $Y:=\{y_1,\dots,y_s\}$ and 
$[k]$ represent the  columns and symbols of $M$. For $j\in[s]$ and $\ell\in [k]$, $y_j\ell$ is an edge in $\Gamma_2$ if and only if $\ell$ is a free  color that is unsaturated in column $j$ of $M$, and  $f_2$ is a mapping on the vertex set of $\Gamma_2$  such that
\begin{align*}
f_2(y_j) = \sum_{\ell \in [k]} \roundup{\frac{\rho_\ell}{n}} - n, \quad
f_2(\ell) = s \roundup{\frac{\rho_\ell}{n}} +(n- s) \rounddown{\frac{\rho_\ell}{n}} - \rho_\ell + b_\ell.
\end{align*}
By a very similar argument, we have $f_2(y_j)\geq 0$ for $j\in [s]$,  $f_2(\ell)\geq 0$ for $\ell\in [k]$,  $\dg_{\Gamma_2}(u)\geq f_2(u)$ for $u\in Y\cup [k]$, and $f_2(Y) = f_2([k])$. Using \eqref{ncond2R} and Ore's theorem, we conclude that $\Gamma_2$ has  an $f_2$-factor $\Theta_2$.

 We let $T$  be a $\min\{n,r+1\}\times \min\{n, s+1\}$ multi-array whose top-left $r\times s$ subarray is identical to $M$, and for  $i\in [r], j\in [s],\ell\in [k]$,
 \begin{align*}
\begin{cases}
\left|\,^{s+1} T_\ell^{i}\right|
    = \roundup{\dfrac{\rho_\ell}{n}} - |M^i_\ell| - \mult_{\Theta_1}(x_i\ell) 
      & \text{if } n > s,\\[1.5mm]
\left|\,^{j} T_\ell^{r+1}\right|
    = \roundup{\dfrac{\rho_\ell}{n}} - |\,^j M_\ell| - \mult_{\Theta_2}(y_j\ell) 
      & \text{if } n > r,\\[1.5mm]
\left|\,^{s+1} T_\ell^{r+1}\right|
    = a_\ell + b_\ell - \rho_\ell + |M_\ell|
      + (2n-r-s)\left\lfloor \dfrac{\rho_\ell}{n} \right\rfloor 
      & \text{if } n > \max\{r,s\}.
\end{cases}
\end{align*}

The  feasibility of  $\{(a_\ell, b_\ell)\}_{\ell\in[k]}$ ensures that $|\leftindex^{s+1} T_\ell^{r+1}|\geq 0$ for $\ell\in [k]$. If $n>s$, each of the first $r$ cells of the last column of $T$ contains $n-s$ symbols, as for $i\in[r]$,
\begin{align*}
   \sum_{\ell\in[k]} |\leftindex^{s+1} T_\ell^{i}|= \sum_{\ell\in[k]}\roundup{\dfrac{\rho_\ell}{n}} - s - \dg_{\Theta_1}(x_i) = n-s.
\end{align*}
Using similar reasoning, each of the first $s$ cells in the last row of $T$ contains $n-r$ symbols if $n>r$. If $n > \max\{r,s\}$, the number of symbols  in cell $(r+1,s+1)$ of $T$ is
\begin{align*}
&\sum_{\ell \in [k]} \Bigl( a_\ell + b_\ell - \rho_\ell + |M_\ell| + (2n-r-s)\rounddown{\dfrac{\rho_\ell}{n}} \Bigr) \\
&\quad = (n-r) \Bigl( n - \sum_{\ell \in [k]} \rounddown{\dfrac{\rho_\ell}{n}} \Bigr)
      + (n-s) \Bigl( n - \sum_{\ell \in [k]} \rounddown{\dfrac{\rho_\ell}{n}} \Bigr) \\
&\quad \quad - n^2 + rs + (2n-r-s) \sum_{\ell \in [k]} \rounddown{\dfrac{\rho_\ell}{n}} \\
&\quad = (n-r)(n-s).
\end{align*}

Let us fix $\ell\in [k]$. We have 
\begin{align*}
|T_\ell| 
&= \sum_{i \in [r]} \Bigl( \Bigl\lceil \frac{\rho_\ell}{n} \Bigr\rceil - |M_\ell^i| - \mathrm{mult}_{\Theta_1}(x_i \ell) \Bigr) 
 + \sum_{j \in [s]} \Bigl( \Bigl\lceil \frac{\rho_\ell}{n} \Bigr\rceil - |\,^j M_\ell| - \mathrm{mult}_{\Theta_2}(y_j \ell) \Bigr) \\[1mm]
&\quad + a_\ell + b_\ell  - \rho_\ell + (2n-r-s) \Bigl\lfloor \frac{\rho_\ell}{n} \Bigr\rfloor + |M_\ell| \\[1mm]
&= \rho_\ell.
\end{align*}
For $i\in[r]$, we have $\lfloor \rho_\ell / n \rfloor \le \lceil \rho_\ell / n \rceil - \mathrm{mult}_{\Theta_1}(x_i \ell) \le \lceil \rho_\ell / n \rceil$, and so 

\begin{align*}
    |T_\ell^i|= |\leftindex^{s+1} T_\ell^{i}|+|M^i_\ell| = \roundup{\dfrac{\rho_\ell}{n}} - \mult_{\Theta_1}(x_i\ell) \approx \dfrac{\rho_\ell}{n},
\end{align*}
and similarly, $|\leftindex^{j} T_\ell|\approx \rho_\ell/n$ for $j\in[s]$. Moreover, the  feasibility of  $\{(a_\ell, b_\ell)\}_{\ell\in[k]}$ implies that
\begin{align*}
|T_\ell^{r+1}| 
&= \sum_{j \in [s]} \Bigl(\Bigl\lceil \frac{\rho_\ell}{n} \Bigr\rceil - |^j M_\ell| - \operatorname{mult}_{\Theta_2}(y_j \ell)\Bigr) 
+ \Bigl(a_\ell + b_\ell - \rho_\ell + |M_\ell| + (2n-r-s)\Bigl\lfloor \frac{\rho_\ell}{n} \Bigr\rfloor \Bigr) \\[1mm]
&= s \Bigl\lceil \frac{\rho_\ell}{n} \Bigr\rceil - \operatorname{deg}_{\Theta_2}(\ell) 
+ \Bigl(a_\ell + b_\ell - \rho_\ell + (2n-r-s)\Bigl\lfloor \frac{\rho_\ell}{n} \Bigr\rfloor \Bigr) \\[1mm]
&= s \Bigl\lceil \frac{\rho_\ell}{n} \Bigr\rceil - \Bigl( s \Bigl\lceil \frac{\rho_\ell}{n} \Bigr\rceil + (n-s) \Bigl\lfloor \frac{\rho_\ell}{n} \Bigr\rfloor - \rho_\ell + b_\ell \Bigr) 
+ \Bigl(a_\ell + b_\ell - \rho_\ell + (2n-r-s)\Bigl\lfloor \frac{\rho_\ell}{n} \Bigr\rfloor \Bigr) \\[1mm]
&= a_\ell + (n-r) \Bigl\lfloor \frac{\rho_\ell}{n} \Bigr\rfloor \\[1mm]
&\approx (n-r) \frac{\rho_\ell}{n},
\end{align*}
and by a very similar argument, $|\leftindex^{s+1} T_\ell| \approx (n-s)\rho_\ell/n$. 
\end{proof}

To complete the proof of Theorem \ref{main2}, it suffices to take $v = n-r$ and $w = n-s$ in the following.

\begin{lemma} \label{claim2}
For $v \in [n-r]$ and $w \in [n-s]$, there exists an $(r+v)\times (s+w)$ multi-array $P$ on $[k]$ whose top-left $r\times s$ subarray coincides with $M$, and such that
\begin{align}\label{Pconditions}
\text{for } i \in [r+v],\ j \in [s+w],\ \ell \in [k], \quad
\begin{cases}
    |P_\ell| = \rho_\ell,\\
    |\,^jP^i| = (p^i)(\,^jp),\\
    \dfrac{|P^i_\ell|}{p^i} \approx \dfrac{|\,^jP_\ell|}{\,^jp} \approx \dfrac{\rho_\ell}{n},
\end{cases}
\end{align}
where
\begin{align*}
p^i &= 
\begin{cases}
    n-r-v+1 & \text{if } i = r+v,\\
    1 & \text{if } i \in [r+v-1],
\end{cases}
&
^jp &= 
\begin{cases}
    n-s-w+1 & \text{if } j = s+w,\\
    1 & \text{if } j \in [s+w-1].
\end{cases}
\end{align*}
\end{lemma}
\begin{proof}
We proceed by induction on $v+w$. The multi-array $T$ from Lemma \ref{claim1} satisfies \eqref{Pconditions}, establishing the base case $v = w = 1$ (so $v+w = 2$). Now we assume that there is an $(r+v)\times (s+w)$ multi-array $P$ satisfying \eqref{Pconditions} for $v+w<2n-r-s$, and suppose that $v<n-r$. We construct an $(r+v+1)\times (s+w)$ multi-array $Q$ on $[k]$ whose top-left $r \times s$ subarray coincides with $M$, and
\begin{align}\label{Qconditions}
    \text{ for }i\in[r+v+1], j\in[s+w], \ell\in[k], \quad 
    \begin{cases}
        |Q_\ell| = \rho_\ell,\\
        |^jQ^i| = (q^i) (^jq),\\
        \dfrac{|Q^i_\ell|}{q^i} \approx \dfrac{|^jQ_\ell|}{^jq} \approx \dfrac{\rho_\ell}{n}, 
    \end{cases}
\end{align}
where
\begin{align*}
q^i &= 
\begin{cases}
    n-r-v & \text{if } i = r+v+1,\\
    1 & \text{if } i \in [r+v],
\end{cases} &
^jq &= 
\begin{cases}
    n-s-w+1 & \text{if } j = s+w,\\
    1 & \text{if } j \in [s+w-1].
\end{cases}
\end{align*}

For $j\in[s+w], \ell\in[k]$, we define the following multi-sets.
\begin{align*}
    \mathbb{P}_\ell &= \bigcup_{j \in [s+w]} \{ (j,\ell) \mid \ell \in \,^jP^{r+v} \}, &
^j\mathbb{P} &= \bigcup_{\ell \in [k]} \{ (j,\ell) \mid \ell \in \,^jP^{r+v} \}.
\end{align*}
Observe that
\begin{align}\label{Hsizeeqn}
    |\mathbb{P}_\ell| &= |P_\ell^{r+v}|, &
    |\,^j\mathbb{P}| &= |\,^jP^{r+v}|.
\end{align}
We define two laminar families of subsets of $\mathbb{P}:=\bigcup\nolimits_{\ell\in[k]}\mathbb{P}_\ell$ as follows.
\begin{align*}
    \mathscr{A} = \{\mathbb{P}_1, \dots, \mathbb{P}_k\}, \quad
    \mathscr{B} = \{\,^1\mathbb{P}, \dots, \,^{s+w}\mathbb{P}\}.
\end{align*}
By Nash-Williams' lemma, there exists a subset $U\subseteq\mathbb{P}$ such that
\begin{align}\label{Zeqn}
    |U\cap W| \approx \dfrac{|W|}{n-r-v+1} \quad \text{ for  }W\in \mathscr{A}\cup\mathscr{B}.
\end{align}
Let $Q$ be the multi-array whose first $r+v-1$ rows match those of $P$, and whose last two rows are formed by splitting the final row of $P$, placing symbols in $U$ into row $r+v$ and all others into row $r+v+1$.

We now verify that $Q$ satisfies \eqref{Qconditions}. Observe that $\,^jp = \,^jq$ for $j \in [s+w]$ and $p^i = q^i$ for $i \in [r+v-1]$. Let us fix a symbol $\ell\in[k]$ and a column $j\in[s+w]$. Then $|Q_\ell| = |P_\ell| = \rho_\ell$ and $|\,^jQ_\ell|/\,^jq = |\,^jP_\ell|/\,^jp \approx \rho_\ell/n$. Since $\,^j\mathbb{P} \in \mathscr{B}$, the construction of $Q$ together with \eqref{Hsizeeqn} and \eqref{Zeqn} gives
\begin{align*}
    |\,^jQ^{r+v}| 
        &= |U \cap \,^j\mathbb{P}| 
        \approx \frac{|\,^j\mathbb{P}|}{\,n-r-v+1\,} 
        = \frac{|\,^jP^{r+v}|}{\,n-r-v+1\,} 
        = \, ^jp 
        = (q^{r+v})(\,^jq),\\[1mm]
    |\,^jQ^{r+v+1}| 
        &= |\,^jP^{r+v}| - |\,^jQ^{r+v}| 
        = (p^{r+v})(\,^jp) - (q^{r+v})(\,^jq) 
        = (q^{r+v+1})(\,^jq).
\end{align*}
Since $\mathbb{P}_\ell \in \mathscr{A}$ and $q^{r+v} = 1$, the construction of $Q$ together with \eqref{Hsizeeqn} and \eqref{Zeqn} implies that
\begin{align*}
    \frac{|Q_\ell^{r+v}|}{q^{r+v}} 
        &= |Q_\ell^{r+v}| 
        = |U \cap \mathbb{P}_\ell| 
        \approx \frac{|\mathbb{P}_\ell|}{\,n-r-v+1\,} 
        = \frac{|P_\ell^{r+v}|}{\,n-r-v+1\,} 
        = \frac{|P_\ell^{r+v}|}{p^{r+v}} 
        \approx \frac{\rho_\ell}{n},\\[1mm]
    |Q_\ell^{r+v+1}| 
        &= |P_\ell^{r+v}| - |Q_\ell^{r+v}| 
        \approx |P_\ell^{r+v}| - \frac{|P_\ell^{r+v}|}{\,n-r-v+1\,} 
        = q^{r+v+1} \frac{|P_\ell^{r+v}|}{p^{r+v}} 
        \approx q^{r+v+1} \frac{\rho_\ell}{n}.
\end{align*}
Hence, $Q$ satisfies \eqref{Qconditions}, and by a similar argument for $w < n-s$, the lemma follows.
\end{proof}

\section{\normalfont{Consequences of the Main Result and Open Problems}} \label{s:corollaries}

It is natural to identify cases in which the conditions of our theorem are simplified; in particular, cases in which the feasible sequence already exists.

For technical reasons, we have assumed throughout the paper that $r,s \in \mathbb{N}$. If this assumption is relaxed to allow $r=s=0$, then Lemma~\ref{claim2} with $r=s=0$ and $v=w=n$ yields a construction of equitable $\brho$-Latin squares. In this case, Lemma~\ref{claim1} is trivial, since one may take $T$ to be a single cell containing $\rho_\ell$ copies of the symbol $\ell$ for each $\ell \in [k]$.
\begin{corollary}
    For  $n,k\in \mathbb{N}$,  $\brho:=(\rho_1,\dots,\rho_k)$ with $1\leq \rho_1,\dots, \rho_k, k\leq  n^2$ and $\sum_{\ell\in [k]} \rho_\ell=n^2$, there exists an equitable $\brho$-Latin square of order $n$. 
\end{corollary}

We next state a result that extends Hall’s theorem and its generalization by Goldwasser et al. \cite{MR3280683} to equitable $\brho$-Latin rectangles.
\begin{corollary}\label{main1}
   An $r \times n$ equitable $\boldsymbol{\rho}$-Latin rectangle $M$ extends to an $n \times n$ equitable $\boldsymbol{\rho}$-Latin square  if and only if the following conditions are met.
\begin{align*}
     &\dfrac{\rho_\ell - |M_\ell|}{n-r} \approx \dfrac{\rho_\ell}{n} &\text{for } \ell\in [k],\\
     &n|J| \geq \sum_{\ell\in[k]}\Big(|J|\roundup{\dfrac{\rho_\ell}{n}} - \min\left\{n\roundup{\dfrac{\rho_\ell}{n}} - \rho_\ell, \eta_J(\ell)\right\} \Big) &\text{for } J\subseteq[s].
\end{align*}
\end{corollary} 
\begin{proof}

    In Theorem \ref{main2}, let $s=n$. Then  \eqref{ncond2R}  for $J\subseteq [s]$ reduces exactly to the second condition above, and  \eqref{MainNcond} becomes $\rho_\ell - |M_\ell| \ge (n-r)\left\lfloor \rho_\ell/n \right\rfloor$ for $\ell\in[k]$,  which is precisely the lower bound in the first condition of this corollary.  As $|M_\ell|\leq r\roundup{\rho_\ell/n}$, the existence of a feasible sequence when $s=n$ reduces to the existence of a sequence $\{(a_\ell,0)\}_{\ell\in[k]}$ such that for $\ell\in[k]$, the following hold.
    \begin{align*}
    \begin{cases}
        a_\ell \geq \rho_\ell - |M_\ell| - (n-r)\rounddown{\dfrac{\rho_\ell}{n}},\vspace{1mm}\\
        a_\ell \leq \min\Big\{n\roundup{\dfrac{\rho_\ell}{n}}, \rho_\ell + \eta_{[r]}(\ell)\Big\} - r\roundup{\dfrac{\rho_\ell}{n}} - (n-r)\rounddown{\dfrac{\rho_\ell}{n}},\vspace{1mm}\\
        \sum_{\ell\in[k]}a_\ell = (n-r)\Big(n - \sum_{\ell\in[k]}\rounddown{\dfrac{\rho_\ell}{n}}\Big).
    \end{cases}
    \end{align*}
    We claim that the sequence $\{(a_\ell,0)\}_{\ell\in[k]}$ such that $a_\ell = \rho_\ell - |M_\ell| - (n-r)\rounddown{\rho_\ell/n}$ for $\ell\in[k]$ is feasible. 
    By the upper bound on $\rho_\ell-|M_\ell|$ given by the first necessary condition of the corollary, we have that
    \begin{align*}
        \rho_\ell - |M_\ell| \leq (n-r)\roundup{\dfrac{\rho_\ell}{n}},
    \end{align*}
    and since $|M_\ell|\geq r\roundup{\rho_\ell/n} - \eta_{[r]}(\ell)$, we have that
    \begin{align*}
        \rho_\ell - |M_\ell| \leq \rho_\ell + \eta_{[r]}(\ell) - r\roundup{\dfrac{\rho_\ell}{n}}.
    \end{align*}
    Moreover, we have that
    \begin{align*}
        \sum_{\ell\in[k]}\Big(\rho_\ell - |M_\ell| - (n-r)\rounddown{\dfrac{\rho_\ell}{n}}\Big) = (n-r)\Big(n-\sum_{\ell\in[k]}\rounddown{\dfrac{\rho_\ell}{n}}\Big).
    \end{align*}
    Thus, this sequence is feasible.

    Finally, note that for $I\subseteq[r]$ and $\ell\in[k]$, we have that $\eta_I(\ell)\leq \eta_{[r]}(\ell) = r\roundup{\rho_\ell/n} - |M_\ell|$. Hence, given the feasible sequence defined above, we have for $\ell\in[k]$ that
    \begin{align}\label{corminimumforcond2}
        \min\Big\{r\Big(\roundup{\dfrac{\rho_\ell}{n}} - \rounddown{\dfrac{\rho_\ell}{n}}\Big) + n\rounddown{\dfrac{\rho_\ell}{n}} - \rho_\ell + a_\ell, \eta_I(\ell)\Big\} &\geq \min\Big\{r\roundup{\dfrac{\rho_\ell}{n}} - |M_\ell|, \eta_I(\ell)\Big\} \geq \eta_I(\ell).
    \end{align}
    In any equitable $r\times n$ $\brho$-Latin rectangle, we have that $\sum_{\ell\in[k]}|M_\ell^i| = n$ and that $|M_\ell^i|\geq \rounddown{\rho_\ell/n}$, and so $\roundup{\rho_\ell/n}-\eta_i(\ell) = |M_\ell^i|$. It follows that for $I\subseteq[r]$ in any equitable $r\times n$ $\brho$-Latin rectangle,
    \begin{align*}
        n|I| = \sum_{i\in I}\sum_{\ell\in[k]}|M_\ell^i| = \sum_{i\in I}\sum_{\ell\in[k]}\Big(\roundup{\dfrac{\rho_\ell}{n}}-\eta_i(\ell)\Big) = \sum_{\ell\in[k]}\Big(|I|\roundup{\dfrac{\rho_\ell}{n}} - \eta_I(\ell)\Big).
    \end{align*}
    By \eqref{corminimumforcond2}, this is equivalent to condition \eqref{ncond2R} of Theorem \ref{main2}, and so the corollary follows directly from our main result.
\end{proof}

The complexity of our main result arises from the presence of free symbols. When every symbol is forced, the necessary conditions simplify considerably. The next two corollaries follow by imposing this restriction on Theorems~\ref{main1} and~\ref{main2}. In particular, setting $\rho_1=\rho_2=\cdots=\rho_k=n$ in the next corollary yields Hall’s theorem.
\begin{corollary} \label{cor:ndivrhoihall}
    When $n \mid \rho_\ell$ for $\ell\in[k]$, any $r\times n$ equitable $\brho$-Latin rectangle can be extended to an $n\times n$ equitable $\brho$-Latin square.
\end{corollary} 
\begin{proof}
As $n \mid \rho_\ell$, both necessary conditions of Corollary \ref{main1} hold trivially. Hence, this result follows directly from Corollary \ref{main1}.
\end{proof}

If we set $\rho_1=\rho_2=\cdots=\rho_k=n$, the condition of the next result reduces to $|M_\ell| \ge r+s-n$ for $\ell\in[k]$. Hence, we recover Ryser’s theorem.
\begin{corollary} \label{cor:ndivrhoiryser}
    When $n\mid\rho_\ell$ for $\ell\in[k]$, an $r\times s$ equitable $\brho$-Latin rectangle $M$ can be extended to an $n\times n$ equitable $\brho$-Latin square if and only if
    \begin{align*}
        \max\Big\{n-r, n-s\Big\}\dfrac{\rho_\ell}{n} \leq \rho_\ell - |M_\ell| \leq (2n-r-s)\dfrac{\rho_\ell}{n} \quad \text{for } \ell\in[k].
    \end{align*}
\end{corollary}
\begin{proof}
    Observe that since $n\mid\rho_\ell$ for $\ell\in[k]$, there are no free colors. Thus as $|M_\ell| \geq (r+s-n)\rho_\ell/n$, the sequence $\{(0, 0)\}_{\ell\in[k]}$ is feasible in this case. Moreover, condition \eqref{ncond2R} holds trivially. Hence, this result follows directly from Theorem \ref{main2}.
\end{proof}

    Let $L_{\brho}(n)$ denote the total number of equitable $n\times n$ $\brho$-Latin squares. When $\brho=(n,n\dots,n)$, $L_{\brho}(n)$ counts the total number of $n\times n$ Latin squares and we know that $(n!)^{2n}/n^{n^2} \leq L_{(n,n,\dots,n)}(n) \leq \Pi_{k=1}^n(k!)^{n/k}$ \cite{courseincomb}. However, there are not yet known bounds on $L_{\brho}(n)$ for more general $\brho$. We propose the following. 
        \begin{question}
            Find good upper and lower bounds for $L_{\brho}(n)$.
        \end{question} 

A {\it Sudoku square} is a $9 \times 9$ Latin square  whose cells are partitioned into nine $3 \times 3$ regions, each containing exactly one copy of each symbol. More generally, let $a,b \in \mathbb{N}$ with $ab = n$ and $\boldsymbol{\rho} = (\rho_1,\dots,\rho_k)$ satisfying $\sum_{\ell \in [k]} \rho_\ell = n^2$. An {\it equitable $(a,b)$-Sudoku $\boldsymbol{\rho}$-Latin square} is an $n \times n$ equitable $\boldsymbol{\rho}$-Latin square whose cells are partitioned into $n$ distinct $a \times b$ regions, each containing either $\lfloor \rho_\ell/n \rfloor$ or $\lceil \rho_\ell/n \rceil$ copies of symbol $\ell \in [k]$. An example of an equitable $(3,4)$-Sudoku $\boldsymbol{\rho}$-Latin square with 
$\boldsymbol{\rho}: = (27,14,24,9,24,12,12,12,10)$ is shown in Figure~\ref{fig:Gerechteex}.    We propose extending Theorem~\ref{main2} to equitable $(a,b)$-Sudoku $\boldsymbol{\rho}$-Latin squares.

\begin{figure}[h]
    \centering
    \begin{tabular}{|cccc|cccc|cccc|}
        \hline
        1 & 2 & 3 & 1 & 4 & 5 & 6 & 3 & 5 & 7 & 8 & 9\\
        4 & 5 & 6 & 3 & 5 & 7 & 8 & 9 & 1 & 2 & 3 & 1\\
        5 & 7 & 8 & 9 & 1 & 2 & 3 & 1 & 4 & 5 & 6 & 3\\
        \hline
        2 & 3 & 1 & 1 & 5 & 6 & 3 & 4 & 7 & 8 & 9 & 5\\
        5 & 6 & 3 & 1 & 7 & 8 & 2 & 5 & 2 & 3 & 1 & 1\\
        7 & 8 & 9 & 5 & 2 & 3 & 1 & 1 & 5 & 6 & 3 & 1\\
        \hline
        3 & 1 & 1 & 2 & 6 & 3 & 4 & 5 & 8 & 9 & 5 & 7\\
        6 & 3 & 4 & 5 & 8 & 9 & 5 & 7 & 3 & 1 & 1 & 2\\
        8 & 9 & 5 & 7 & 3 & 1 & 1 & 2 & 6 & 3 & 4 & 5\\
        \hline
        1 & 1 & 2 & 3 & 3 & 4 & 5 & 6 & 2 & 5 & 7 & 8\\
        3 & 1 & 5 & 6 & 9 & 5 & 7 & 8 & 1 & 1 & 2 & 3\\
        9 & 5 & 7 & 8 & 1 & 1 & 2 & 3 & 3 & 4 & 5 & 6\\
        \hline
    \end{tabular}
    \caption{An equitable $(3,4)$-Sudoku $\boldsymbol{\rho}$-Latin square}
    \label{fig:Gerechteex}
\end{figure}

\begin{question}
Determine the conditions under which a given equitable $\boldsymbol{\rho}$-Latin rectangle can be extended to an equitable $(a,b)$-Sudoku $\boldsymbol{\rho}$-Latin square.
\end{question}
    \noindent To the best of our knowledge, this problem is still open for the case where $\brho=(n,\dots,n)$.

A {\it Gerechte framework} is a partition of an $n \times n$ array into $n$ regions $R_1,\dots,R_n$, each containing $n$ cells. It is equitably $\boldsymbol{\rho}$-realizable if there exists an equitable $\boldsymbol{\rho}$-Latin square $M$ on  $[k]$ such that each  $\ell \in [k]$ occurs either $\lfloor \rho_\ell/n \rfloor$ or $\lceil \rho_\ell/n \rceil$ times in each region $R_i$ for $i\in[n]$. In this case, $M$ is an equitable $\boldsymbol{\rho}$-realization, and the pair of framework and $M$ is an equitable $\boldsymbol{\rho}$-Gerechte design. When $\boldsymbol{\rho}=(n,\dots,n)$, we drop the qualifier, and a Sudoku square is a Gerechte design with $n=9$ and $3 \times 3$ regions \cite{Courtiel_2011, MR2408485}. It is nontrivial that any Gerechte framework for which each region is either an $s\times t$ or a $t\times s$ rectangle has a realization \cite{Courtiel_2011}. We propose the following.
\begin{question}
For a given $\brho = (\rho_1,\rho_2,\dots,\rho_k)$ where $\sum_{\ell\in[k]}\rho_\ell = n^2$, which Gerechte frameworks are equitably $\brho$-realizable?
\end{question}

 Gerechte designs were motivated by agricultural experiments using Latin squares, where distinct symbols represented experimental parameters \cite{MR2408485}. The Gerechte framework ensures fair testing despite variations in soil across a field, and orthogonality accounts for the influence of one experiment on the next. Beyond agriculture, Gerechte designs have applications in   projective geometry and coding theory. Bailey, Cameron, and Connelly provide a clear overview of both the motivations and related mathematical structures \cite{MR2408485}.

\bibliographystyle{plain}
\bibliography{Latinizedryser}
\end{document}